\documentclass[a4paper,11pt]{article}

\usepackage[T2A]{fontenc}
\usepackage[cp1251]{inputenc}
\usepackage[english,ukrainian]{babel}
\usepackage{amssymb}
\usepackage{amsmath,amsthm, enumerate}
\usepackage{multicol}
\usepackage{geometry}

\usepackage{fancyhdr}
\newcommand{\norm}[1]{\left\Vert#1\right\Vert}
\newcommand{\abs}[1]{\left\vert#1\right\vert}
\newcommand{\set}[1]{\left\{#1\right\}}
\newcommand{\R}{\mathbb R}

\newcommand{\eps}{\varepsilon}
\newcommand{\ex}[1]{\mathsf{E}\left[\,#1\,\right]}
\newcommand{\Ex}[1]{\mathsf{E}\Big[\,#1\,\Big]}
\newcommand{\ind}[1]{\mathbf{1}_{#1}}
\newcommand{\wotimes}{\mathop{\widehat\otimes}}
\newcommand{\mH}{\hat H}
\newcommand{\MH}{\check H}

\fancypagestyle{plain}{%
\lfoot{\selectlanguage{ukrainian}\footnotesize \copyright\ Г.М. Шевченко, 2013}%
}

\makeatletter
\renewcommand\section{\@startsection{section}{1}{\z@}%
                                   {3.5ex \@plus 1ex \@minus .2ex}%
                                   {2.3ex \@plus .2ex}%
                                   {\centering\normalfont\bfseries}}

\def\@biblabel#1{#1.}
\makeatother

\theoremstyle{plain}
\newtheorem{theorem}{Theorem}
\newtheorem{lemma}{Lemma}
\newtheorem{proposition}{Proposition}

\theoremstyle{definition}
\newtheorem{definition}{Definition}
\theoremstyle{remark}
\newtheorem{remark}{Remark}

\let\le\leqslant

\let\ge\geqslant

\begin{document}
\thispagestyle{plain}

УДК 519.21

\setlength{\columnsep}{7mm}%
\setlength{\columnseprule}{.4pt}
\begin{multicols}{2}

\begin{flushleft}
Г.М. Шевченко$^1$, к. ф.-м. н., доцент

\vspace{0.2 cm}

{\bfseries  Локальні часи для мультидробових квадратично гауссівских процесів}
\end{flushleft}
\medskip

{\itshape У статті розглядаються мультидробові процеси, які задаються подвійним інтегралом Іто--Вінера з певним ядром і є узагальненнями 
мультидробового процесу Розенблатта. Для таких процесів доведено неперервність та існування квадратично інтегрованого локального часу.

\smallskip

Ключові слова: процес Розенблатта, мультидробовий процес, локальний час.}

$^1$Київський національний університет імені Тараса Шевченка, механіко-математичний факультет, 01601, Київ, вул.~Володимирська 64, e\nobreakdash-mail: zhora@univ.kiev.ua

\columnbreak

{\selectlanguage{english}

\begin{flushleft}
G.M. Shevchenko$^1$, Ph.D.,  Associate Prof.

\vspace{0.2 cm}

{\bfseries Local times for multifractional square Gaussian processes}
\end{flushleft}
\medskip

{\itshape We consider multifractional process given by double It\^o--Wiener integrals, which generalize the multifractional Rosenblatt process. We prove that this process is continuous and has a square integrable local time.

\smallskip

 \vfill}

$^1$Taras Shevchenko National University of Kyiv, Department of Mechanics and Mathematics, 01601, Kyiv, Volodymyrska str., 64, e\nobreakdash-mail: zhora@univ.kiev.ua
}

\end{multicols}

\bigskip
Communicated by Prof.  Kozachenko Yu.V.

\selectlanguage{english}
\setlength{\columnsep}{4mm}%
\setlength{\columnseprule}{0pt}
\begin{multicols}{2}

\section*{Introduction}

Fractional processes play an important role in modeling the long-range dependence property. In view of this, they find numerous applications, most notably in financial mathematics, geophysics, hydrology, telecommunication research etc. The most popular among fractional processes is the Gaussian one, viz.\ the fractional Brownian motion. However, it has several drawbacks. Firstly, its increments are stationary, which does not allow to model processes whose properties vary essentially as time flows. Secondly, it is selfsimilar, so has similar properties on different time scales, yet only few real-world processes have this property. Lastly, it has Gaussian distribution, so its tails are extremely light. 

In view of these drawbacks, several authors considered some generalizations of fractional processes. To address the first two drawbacks, it is useful to introduce some multifractional processes, by letting the memory parameter to vary with time. On this way, some multifractional counterparts of fractional Brownian motion were proposed in \cite{BenassiJaffardRoux97,PeltierLevyVehel,ralchenko},
to name only a few. These processes are also Gaussian, so they do not solve the light tails problem. Therefore, in \cite{shevrosen}, a multifractional analogue of Rosenblatt process was defined. The Rosenblatt process has a generalized chi-square distribution, so its tails are heavier than Gaussian ones. The Rosenblatt process first appeared in \cite{rosenblatt}. It was investigated in many articles, we cite only  \cite{tudor-torres}, where some applications of the Rosenblatt process in financial modeling were considered. To allows for more flexibility, recently in \cite{tudor-maejima} a generalization of Rosenblatt process was proposed; the authors also proved that this process appears naturally as a limit of some normalized cross-variation of two different fractional Brownian motions. In the present paper we consider a multifractional analogue of the generalized Rosenblatt process. 

The paper is organized as follows. In Section 1, we give necessary information about double stochastic integrals and local times. In Section 2, we introduce the main object, the multifractional generalized Rosenblatt process, and show that it has a H\"older continuous modification. Finally, in Section 3, we prove that this process has a square integrable local time.

\section{Preliminaries}

Throughout the article, the symbol $C$ will denote a generic constant whose value is not important and may change from line to line.

First we give provide essential information on double stochastic integrals. More information can be found in  \cite{nualart}.

Let $W$ be the standard Wiener measure on the real line, i.e. a random measure with independent values on disjoint sets,
such that for a Borel set $A$ of a finite Lebesgue measure $\lambda(A)$, $W(A)$ has a centered Gaussian distribution with the variance $\lambda(A)$.

The double stochastic It\^o--Wiener integral is constructed as follows. Let $\widehat L^2(\R^2)$ be the Hilbert space of symmetric square integrable functions on $\R^2$ with the norm $\norm{\cdot}$. Suppose  $f\in \widehat L^2(\R^2)$ is a symmetric simple function of the form
\begin{equation}
\label{simple}
f(x,y) = \sum_{k,j=1}^n a_{k\,j}\ind{A_k}(x)\ind{A_j}(y),
\end{equation}
where $A_1,\dots,A_n$ are disjoint sets in $\R$ of a finite Lebesgue measure, $a_{k\,j}=a_{j\,k}$ for all $k,j=1,\dots,n$, $a_{k\,k}=0$ for all $k=1,\dots,n$. Then we define
\begin{gather*}
I_2(f) = \iint_{\R^2} f(x,y) W(dx)\, W(dy) \\= \sum_{k,j=1}^n a_{k\,j}W(A_k)W(A_j).
\end{gather*}
The map $I_2$ can be extended to a linear isometry on $\widehat L^2(\R^2)$. Its properties are the following.

1. Let $a,b\in L^2(\R)$ be orthogonal and $f=a\wotimes b\in \widehat L^2(\R^2)$ be their symmetric tensor product, that is
$f(x,y) =  \big(a(x)b(y) + a(y)b(x)\big)/2$. Then $I_2(h) = I_1(a)I_1(b)$.

2. For all $a\in L^2(\R)$, $I_2(a\wotimes a) = (I_1(a))^2 - \norm{a}^2_{L^2(\R)}$.

3. Define, for a function $f\in \widehat L^2(\R^2)$, a Hermitian operator 
$$
(M_f a)(x) = \int_\R f(x,y) a(y) dy.
$$
Let $\lambda_{k,f}$, $k\ge 1$ be its eigenvalues ordered by absolute value (largest first) and  
$\varphi_{k,f}$, $k\ge 1$ be the corresponding eigenfunctions. Then, since
$$
f(x,y) = \sum_{k\ge 1}\lambda_{k,f} \varphi_{k,f}\wotimes\varphi_{k,f},
$$
we have
$$
I_2(f) = \sum_{k\ge 1} \lambda_{k,f} (\zeta^2_k-1),
$$
with $\zeta_k=I_1(\varphi_{k,f})$ being independent standard Gaussians. Therefore, the characteristic function of $I_2(f)$ is
\begin{equation}
\begin{gathered}
\label{char1}
\ex{e^{i\alpha I_2(f)}} = \prod_{k\ge 1} \ex{e^{i\alpha \lambda_{k,f}(\zeta_k^2-1)}} \\= \prod_{k\ge 1} \frac{e^{-i\alpha \lambda_{k,f}}}{\sqrt{1-2i\alpha \lambda_{k,f}}},
\end{gathered}
\end{equation}
and it admits the estimate
\begin{equation}
\begin{gathered}
\label{charineq}
\abs{\ex{e^{i\alpha I_2(f)}}} = \left(\prod_{m\ge 1} \big(1+4\alpha^2 \lambda^2_{k,f}\big)\right)^{-1/4}\\
\le \bigg(1+4\alpha^2 \sum_{k\ge 1}\lambda^2_{k,f} + 16\alpha^4 \sum_{j<k}\lambda^2_{j,f}\lambda^2_{k,f}\\ + 64 \alpha^6 \sum_{j<k<l}\lambda^2_{j,f}\lambda^2_{k,f}\lambda^2_{l,f}\bigg)^{-1/4}.
\end{gathered}
\end{equation}
Finally, we mention that by \cite[Corollary 7.36]{janson}, for each $m\ge 2$ there exists a constant $C_m$ such that 
for every $f\in\widehat L^2(\R^2)$, 
\begin{equation}\label{I2f-moments}
\ex{I_2(f)^m} \le C_m \left(\ex{I_2(f)^2}\right)^{m/2}.
\end{equation}

Further we give some basics about local times.

Let $\{Z_t, t\ge 0\}$ be a separable random process. The \textit{occupation measure} for $Z$ is defined as follows: for Borel $A\subset [0,\infty)$,  $B \subset\R$,
\begin{equation*}
\mu(A,B) = \lambda(\{s\in A, Z_s \in B\}),
\end{equation*}
and the \textit{local time} is the Radon--Nikodym derivative
$L(A,x) = \frac{d\mu(A,\cdot)}{d\lambda}(x)$.

We need the following proposition from \cite{berman}.
\begin{proposition}\label{bermcond}
A process $\set{Z_t,t\ge 0}$ has a local time $L(A,x)\in L^2(\R\times \Omega)$ iff
\begin{equation}
\label{ltx^2}
\int_\R\int_A\int_A \ex{e^{iv(Z_t-Z_s)}}ds\,dt\,dv<\infty.
\end{equation}
\end{proposition}

\section{Generalized multifractional Rosenblatt process and its basic properties}

The generalized Rosenblatt process was introduced in \cite{tudor-maejima}. For $H_1,H_2\in(1/2,1)$, define the kernel
\begin{gather*}
K_{H_1,H_2}(t,x,y) = (s-x)_+^{H_1/2-1}(s-y)_+^{H_2/2-1}\\
+ (s-x)_+^{H_2/2-1}(s-y)_+^{H_1/2-1} .
\end{gather*}
\begin{definition}
For  $H_1,H_2\in (1/2,1)$, the process
\begin{equation*}
Y_t^{H_1,H_2} = \iint_{\R^2} \int_{0}^t K_{H_1,H_2}(s,x,y) ds\, W(dx)W(dy)
\end{equation*}
is the \textit{generalized Rosenblatt process} with parameters $H_1$ and $H_2$.
\end{definition}
It was shown in \cite{tudor-maejima} that $Y^{H_1,H_2}$ is an $H/2$-selfsimilar process with stationary increments.

\begin{remark}
In fact, the definition in \cite{tudor-maejima} differs from the one here by some normalizing constant. For the sake of technical simplicity, we omit this constant.
\end{remark}

Now we define a multifractional counterpart of the generalized Rosenblatt process. Let, for a fixed $T>0$,  $H_i\colon[0,T]\to (1/2,1)$, $i=1,2$, be continuous functions. 

\begin{definition}
The process
\begin{gather*}
X_t = Y_t^{H_1(t),H_2(t)} \\
= \iint_{\R^2} \int_{0}^t K_{H_1(t),H_2(t)}(s,x,y) ds\, W(dx)W(dy)
\end{gather*}
will be called a \textit{multifractional generalized Rosenblatt process}. 
\end{definition}
Denote for $i=1,2$ \ $\mH_i = \min_{t\in[0,T]} H_i(t)$, $\MH_i = \max_{t\in[0,T]} H_i(t)$. Let also $\mH= (\mH_1+ \mH_2)/2$, $\MH= (\MH_1+\MH_2)/2$.

The following hypotheses about the functions $H_1$ and $H_2$ will be assumed throughout the article: there exists  some $\gamma>\MH$ such that
\begin{equation*}
\abs{H_i(t)-H_i(s)}\le \abs{t-s}^\gamma
\end{equation*}
for $i=1,2$ and $t,s\in [0,T]$.

Further we establish some continuity properties of the process $X$. 
We start with the following lemma, which provides continuity of the process $Z^{H_1,H_2}$ in its parameters. 
\begin{lemma}\label{lem:Hcont}
For all $H^{\vphantom{'}}_1,H_1'\in[\mH_1,\MH_1],H^{\vphantom{'}}_2,H_2'\in[\mH_2,\MH_2]$, $t\in[0,T]$
\begin{equation*}
\Ex{\big({Y^{H^{\vphantom{'}}_1,H^{\vphantom{'}}_2}_t - Y^{H_1',H_2'}_t}\big)^2} \le C\sum_{i=1}^2\big({H^{\vphantom{}}_i-H_i'}\big)^2.
\end{equation*}
\end{lemma}
\begin{proof}
Start by estimating
\begin{gather*}
\Ex{\big({Y^{H^{\vphantom{'}}_1,H^{\vphantom{'}}_2}_t - Y^{H_1',H_2'}_t}\big)^2} \\\le 
2 \bigg(\Ex{\big({Z^{H^{\vphantom{'}}_1,H^{\vphantom{'}}_2}_t - Y^{H_1',H_2^{\vphantom{'}}}_t}\big)^2} \\+ \Ex{\big({Z^{H_1',H^{\vphantom{'}}_2}_t - Y^{H_1',H_2'}_t}\big)^2}\bigg).
\end{gather*}
By symmetry, it is enough to show that
 \begin{equation*}
\ex{\big({Y^{H^{\vphantom{'}}_1,H^{\vphantom{'}}_2}_t - Y^{H_1',H_2^{\vphantom{'}}}_t}\big)^2} \le C\big({H^{\vphantom{}}_1-H_1'}\big)^2.
\end{equation*}
Denote $$f_{H_1,H_2}(s,x,y) = (s-x)_+^{H_1/2-1}(s-y)_{+}^{H_2/2-1}.$$
Write
\begin{gather*}
\ex{\big({Y^{H^{\vphantom{'}}_1,H^{\vphantom{'}}_2}_t - Y^{H_1',H_2'}_t}\big)^2}\\
=\ex{\left(\iint_{\R^2}\int_0^t \Delta(s,x,y) ds\, W(dx)W(dy)\right)^2}\\
= \iint_{\R^2}\left(\int_0^t \Delta(s,x,y) ds\right)^2 dx\,dy,
\end{gather*}
where 
\begin{gather*}
\Delta(s,x,y)=K_{H^{\vphantom{'}}_1,H^{\vphantom{'}}_2}(s,x,y) - K_{H'_1,H^{\vphantom{'}}_2}(s,x,y)\\
= \frac{1}2 \int^{H_1}_{H_1'} \big(f_{h,H_2}(s,x,y) \log (s-x)_+ 
\\ + f_{H_2,h}(s,x,y) \log (s-y)_+\big)dh.
\end{gather*}
Now take some $\eps\in\big(0,(1-\MH_1)\wedge (\mH_1-1/2)\big)$. Since $\log x \le C (x^\eps + x^{-\eps})$, we can estimate
\begin{gather*}
\abs{\Delta(s,x,y)} \le C\big|{H^{\vphantom{'}}_1-H_1'}\big|
\big(f_{\mH_1-\eps,H_2}\\
+f_{\MH_1+\eps,H_2}+f_{H_2,\mH_1-\eps,H_2}
+f_{H_2,\MH_1+\eps}\big) (s,x,y)\\
= C\big|{H^{\vphantom{'}}_1-H_1'}\big|\big(
K_{\mH_1-\eps,H_2} + K_{\MH_1+\eps,H_2}
\big)(s,x,y).
\end{gather*}
For any $H_1,H_2\in(1/2,1)$, $t\in[0,t]$, $x,y\in\R$
\begin{gather*}
\iint_{\R^2}\left(\int_0^t K_{H_1,H_2} (s,x,y)ds\right)^2 dx\,dy\\ = \Ex{\big(Y^{H_1,H_2}_t\big)^2} <\infty.
\end{gather*}
moreover, this expression is continuous in $H_1$ and $H_2$ (see \cite[Lemma 2.2]{tudor-maejima}), hence, bounded for $H_2\in[\mH_2,\MH_2]$ and $H_1\in \big\{\mH_1-\eps,\MH_1+\eps\big\}$. Consequently,
\begin{gather*}
\Ex{\big({Y^{H^{\vphantom{'}}_1,H^{\vphantom{'}}_2}_t - Y^{H_1',H_2'}_t}\big)^2}
\le  C\big|{H^{\vphantom{'}}_1-H_1'}\big|,
\end{gather*}
as required.
\end{proof}
Now we are able to formulate result about the generalized multifractional Rosenblatt process $X$.
\begin{proposition}\label{prop:ex(xt-xs)^2}
There exist positive constants $C_1$ and $C_2$ such that for all $s,t\in[0,T]$ sufficiently close,
$$
 \ex{(X_t-X_s)^2}\ge C_1\abs{t-s}^{H_1(t)+ H_2(t)}
$$
and 
$$
 \ex{(X_t-X_s)^2}\le C_2\abs{t-s}^{H_1(t)+ H_2(t)}.
$$
Consequently, for all $s,t\in[0,T]$
$$
\ex{(X_t-X_s)^2}\ge C_1'\abs{t-s}^{2\MH}
$$
and 
$$
\ex{(X_t-X_s)^2}\le C_2'\abs{t-s}^{2\mH}
$$
with some positive constants $C_1'$ and $C_2'$.
\end{proposition}
\begin{proof}
From the self-similarity and stationary increments property of $Y^{H_1,H_2}$, we have
$$
\Ex{\big(Y^{H_1,H_2}_t - Y^{H_1,H_2}_s\big)^2} = K\abs{t-s}^{H_1 + H_2}
$$
with $K = \ex{\big(Y_1^{H_1,H_2}\big)^2}$. 

For brevity, denote $Y^{\overline H(t)} = Y^{H_1(t),H_2(t)}$.
Write 
\begin{gather*}
\ex{(X_t-X_s)^2} = \Ex{\big(Y^{\overline H(t)}_t - Y^{\overline H(s)}_s\big)^2}\le\\
 2\Ex{\big(Y^{\overline H(t)}_t - Y^{\overline H(t)}_s\big)^2}+2\Ex{\big(Y^{\overline H(t)}_s - Y^{\overline H(s)}_s\big)^2} \\
= K\abs{t-s}^{H_1(t) + H_2(t)} + \Ex{\big(Y^{\overline H(t)}_s - Y^{\overline H(s)}_s\big)^2}.
\end{gather*}
By Lemma~\ref{lem:Hcont}, 
\begin{gather*}
\Ex{\big(Y^{\overline H(t)}_s - Y^{\overline H(s)}_s\big)^2}\le 
C\sum_{i=1}^2 \big|H_i(t)-H_i(s)\big|^2\\
\le C\abs{t-s}^{2\gamma} = o(\abs{t-s}^{2\MH}), \abs{t-s}\to 0,
\end{gather*}
whence we derive the required statement.
\end{proof}
As a corollary, we get H\"older continuity of the generalized multifractional Rosenblatt process. 
\begin{theorem}
For any $\alpha<\mH$, the process $X$ has a modification, which is  H\"older continuous of order $\alpha$.
\end{theorem}
\begin{proof}
Using Proposition~\ref{prop:ex(xt-xs)^2} and the estimate \eqref{I2f-moments}, we get 
$$
\ex{(X_t-X_s)^m}\le C\abs{t-s}^{m\mH}
$$
for all $m\ge 1$. Taking $m>(\mH-\alpha)^{-1}$, we obtain the statement from the Kolmogorov--Chentsov theorem.
\end{proof}
\section{Existence of local time for generalized Rosenblatt process}
In this section we prove the existence of local time for both generalized Rosenblatt process and multifractional generalized Rosenblatt process. 
\begin{theorem}\label{thm:lt}
For each $H_1,H_2\in (1/2,1)$, the process $Y^{H_1,H_2}$ has a square integrable local time on $[0,T]$.
\end{theorem}
\begin{proof}
By the Proposition \ref{bermcond} it is enough to check the condition \eqref{ltx^2} for $A=[0,T]$. Fix some $H_1,H_2\in (1/2,1)$ and abbreviate $Y = Y^{H_1,H_2}$, $H=(H_1+H_2)/2$. Since $Y$ is  $(H_1+H_2)/2$-selfsimilar and has stationary increments,
\begin{gather*}
\ex{e^{iz(X_t-X_s)}} = \ex{e^{izX_{\abs{t-s}}}} = \ex{e^{iz\abs{t-s}^H X_1}}.
\end{gather*}
We have $X_1 = I_2(\phi)$ with $$\phi(x,y)= \int_0^1 K_{H_1,H_2}(s,x,y)ds.$$ Therefore, by \eqref{charineq}
\begin{gather*}
\ex{e^{iz\abs{t-s}^H X_1}} = \ex{e^{iz\abs{t-s}^H I_2(\phi)}}\\
\le \left(1+64 z^6 \abs{t-s}^{6H}\lambda^2_{1,\phi}\lambda^2_{2,\phi}\lambda^2_{3,\phi}\right)^{-1/4}.
\end{gather*}
Now we will prove that $\operatorname{dim} M_\phi L^2(\R) =\infty$. Take any $a>0$ and set $f_a(x) = e^{ax}\ind{x\le 1}$. For any $h>0$, 
\begin{gather*}
\int_{\R}(s-y)_+^{h/2-1} f_a(y) dy = e^{as} \int_0^\infty t^{h/2-1}e^{-at}dt \\= a^{-h/2}\Gamma(h/2)e^{as}=: c_{a,h}e^{as}.
\end{gather*}
Therefore,
\begin{gather*}
(M_\phi f_\alpha)(x) = \int_{\R}\int_0^1 K_{H_1,H_2}(s,x,y)f_\alpha(y)dy=
\\
 \int_0^1 \left(c_{a,H_1}(s-x)_+^{\alpha_1}+c_{a,H_2}(s-x)_+^{\alpha_2}\right)e^{as}ds,
\end{gather*}
where $\alpha_i = H_i/2-1$, $i=1,2$.
In particular, $(M_\phi f_a)(x)$ decays as $e^{ax}\abs{x}^{\alpha_1 \vee \alpha_2}$ for  $x\to-\infty$. Hence the claim $\operatorname{dim} M_\phi L^2(\R) =\infty$ follows obviously.

Consequently, $\lambda = \lambda^2_{1,\phi}\lambda^2_{2,\phi}\lambda^2_{3,\phi}>0$, and 
\begin{gather*}
\int_\R\int_0^T\int_0^T \ex{e^{iz(X_t-X_s)}}ds\,dt\,dz\\
\le  \int_0^T\int_0^T\int_\R \left(1+64\lambda z^6\abs{t-s}^{6H}\right)^{-1/4}dz\,ds\,dt\\
= \int_0^T\int_0^T  \int_\R \frac{dz}{\left(1+\lambda z^6\right)^{1/4}}\frac{ds\,dt}{2\abs{t-s}^{H}}<\infty,
\end{gather*}
as required.
\end{proof}
From this theorem we can deduce existence of local time for multifractional generalized process exactly the same way as \cite[Theorem 3.2]{shevrosen} is deduced from \cite[Theorem 3.1]{shevrosen}. Nevertheless, for completeness we present the proof. 
\begin{theorem}
The multifractional generalized Rosenblatt process $X$ has a square integrable local time on $[0,T]$.
\end{theorem}
\begin{proof}
In view of the additivity of the local time in $A$, it is enough to check the assumption \eqref{ltx^2} for $A=[s,t]$, with $t-s$ small enough. 

To this end, denote $Y^{\overline H}_t = Y^{H_1,H_2}_t$,
$$
f_{\overline H}(t,x,y) = \int_0^t K_{H_1,H_2}(u,x,y)du
$$
and write 
\begin{gather*}
\abs{\ex{e^{iz(X_t-X_s)}}} = \abs{\ex{e^{iz I_2(g)}}}\\ \le \left(1+64 z^6 \lambda^2_{1,g}\lambda^2_{2,g}\lambda^2_{3,g}\right)^{-1/4},
\end{gather*}
where $g(x,y) = f_{\overline H(t)}(t,x,y)-f_{\overline H(s)}(s,x,y)$ is the kernel corresponding to $X_t-X_s$. It can be represented as $g_1(x,y)+g_2(x,y)$, where 
\begin{align*}
g_1(x,y)&= f_{\overline H(t)}(t,x,y)-f_{\overline H(t)}(s,x,y),\\ g_2(x,y)&= f_{\overline H(t)}(s,x,y)-f_{\overline H(s)}(s,x,y).
\end{align*}
Evidently, \begin{equation}\label{lakg}
\abs{\lambda_{k,g}}\ge \abs{\lambda_{k,g_1}} - \norm{M_{g_2}}\ge \abs{\lambda_{k,g_1}} - \norm{g_2}.
\end{equation}
It follows from  Lemma~\ref{lem:Hcont} that $\norm{g_2}\le C(t-s)^\gamma$. 
On the other hand,
$I_2(g_1)=Y_t^{\overline H(t)}-Y_s^{\overline H(t)}$ is distributed as $(t-s)^{H(t)}Y_1^{\overline H(t)}$ with $H(t) = (H_1(t)+H_2(t))/2$. Therefore,
$$\lambda_{k,g_1} = (t-s)^{H(t)}\lambda_{k,\phi_{\overline H(t)}}\ge C (t-s)^{\MH}\lambda_{k,\phi_{\overline H(t)}},$$
where $\phi_{\overline H(t)}$ is the kernel corresponding to $Y_1^{\overline H(t)}$. \selectlanguage{ukrainian} From the proof of Theorem~\ref{thm:lt} it follows that $\lambda_{k,\phi_{H(t)}}>0$,
$k=1,2,3$. In view of continuity, $\lambda_{k,g_1} \ge C (t-s)^{\MH}$ for $k=1,2,3$. Since $\abs{\lambda_{k,f_1}-\lambda_{k,f_2}}\le \norm{f_1-f_2}$, then  the estimates $\norm{g_2}\le C(t-s)^\gamma = o((t-s)^{\MH})$ with \eqref{lakg} yield $\abs{\lambda_{k,g}}\ge C (t-s)^{\MH}$, $k=1,2,3$, whenever $t-s$ is small enough. Thus,
\begin{gather*}
\ex{e^{iz(X_t-X_s)}} \le \left(1+ C z^6 (t-s)^{6\MH} \right)^{-1/4},
\end{gather*}
and the rest of proof goes as for Theorem~\ref{thm:lt}.
\end{proof}

\begin{flushright}
Received: 1.07.2013
\end{flushright}
\end{multicols}
\end{document}